\documentclass[11pt]{amsart}
\usepackage{amssymb, amsmath, amsthm}
\usepackage{amsrefs}
\usepackage{graphicx}
\usepackage[final]{hyperref}
\usepackage[a4paper, centering]{geometry}

\usepackage{color}
\usepackage{graphicx}

\newtheorem{theorem}{Theorem}
\newtheorem{proposition}[theorem]{Proposition}
\newtheorem{lemma}[theorem]{Lemma}

\theoremstyle{definition}

\theoremstyle{remark}
\newtheorem{remark}[theorem]{Remark}

\def\R{\mathbb{R}}

\def\N{\mathbb{N}}

\def\pscal#1#2{\left\langle#1,\,#2\right\rangle}

\def\Om{\Omega} 
\def\k{\kappa} 

\def\subjet{J^{2,-}}

\DeclareMathOperator{\dist}{dist}

\DeclareMathOperator{\Tr}{Tr}

\begin{document}

\title[]%
{Power-logconcavity of the Laplacian ground state} 

\author[G.~Crasta, I.~Fragal\`a]{Graziano Crasta,  Ilaria Fragal\`a}

\address[Graziano Crasta]{Dipartimento di Matematica ``G.\ Castelnuovo'',
Sapienza University of Rome\\
P.le A.\ Moro 5 -- 00185 Roma (Italy)}
\email{graziano.crasta@uniroma1.it}

\address[Ilaria Fragal\`a]{
Dipartimento di Matematica, Politecnico\\
Piazza Leonardo da Vinci, 32 --20133 Milano (Italy)
}
\email{ilaria.fragala@polimi.it}



\date{\today}

\begin{abstract}  Let $u$ be the first Dirichlet Laplacian eigenfunction of a bounded convex set $\Om$ in $\R ^n$.  We strengthen the classical result by Brascamp-Lieb which asserts that $u$ is logconcave in $\Om$: we prove that, if $u$ is normalized so that its $L^\infty$-norm  does not exceed a threshold 
$\overline \k (\Om)<1$ depending explicitly on the diameter of the domain and on its principal frequency, the function $- ( - \log u ) ^{1/2}$ is concave in $\Om$.  

 \end{abstract} 
\maketitle

\section{Introduction}

A seminal result by Brascamp-Lieb states that logconcavity is preserved by the heat flow $v _ t = \Delta v$ with Dirichlet boundary conditions
 on an open bounded convex domain $\Om \subset \R ^n$ \cite{BL}: as a consequence of the logconcavity of the heat kernel, 
 if the initial datum $v _0(x) = v (x, 0)$ is  a bounded, positive, logconcave function in $\Om$, then  $v (\cdot, t)$  remains logconcave for every $t>0$
(in fact,  Lee-V\'azquez   later proved 
   that logconcavity arises spontaneously  along the  Dirichlet heat flow even when the initial datum is not logconcave, see \cite{LV}).  
   
The logconcavity of the ground state, namely the first Dirichlet eigenfunction $u$,  follows by a limiting argument, since
 log-concavity is stable under scalar multiplication and, for some $c>0$, 
 $$\lim _{ t \to + \infty} e ^ { \lambda _ 1 t} v  ( t, x) = c u ( x)  \quad \text{  uniformly in  }  \overline \Om\,.$$
This result has several extensions (see e.g.\ \cite{Sak,  BS09, CF20, CFrobin, Col25}) and plenty of  important implications, for instance in the study of the fundamental gap, concentration phenomena,  Poincar\'e or Brunn-Minkowski type inequalities  
\cite{ABF1, AC11,  BorBM, Col2005, ledoux}.

In a series of recent papers \cite{IST20,IST22,IST24}, Ishige-Salani-Takatsu raised and investigated  the question whether the above properties remain valid for some stronger kind of concavity beyond logconcavity. 
In particular, they introduced the following notion of power-logconcavity  (see \cite[Definition 2.1]{IST20}):
given $\alpha \in (0, 1]$, a nonnegative bounded function $u$ is said to be {\it $\alpha$-logconcave}  in $\Om$  if, letting 
\begin{equation}\label{f:lalfa} 
L _\alpha (s):=   - ( - \log s ) ^{\alpha} \,,
\end{equation} 
for every  sufficiently small $\k>0$ one has
\begin{equation}\label{f:alfaconc} L _\alpha (\k u ( ( 1-t )  x + t   y) \geq ( 1-t ) L _\alpha (\k u (x)) + t L _\alpha (\k u (y))
 \qquad \forall x, y \in \Om, \ \forall t \in [0, 1]\, .
\end{equation}   
If $u$ takes values into $[0, 1]$, the above definition amounts to requiring that 
\eqref{f:alfaconc} be satisfied  for some $\k \in (0, 1]$, as it then automatically remains valid for all smaller $\k$, see \cite[Lemma 2.1]{IST20}). 
Moreover, it is straightforward to see that, if $u$ is $\alpha$-logconcave, it is also $\beta$-logconcave  for every $\beta>\alpha$. 
Thus, for every $\alpha \in (0, 1)$, $\alpha$-logconcavity is stronger than logconcavity (which corresponds to $\alpha = 1$). 

The  choice of the function $L _\alpha$ in \eqref{f:lalfa} was motivated by the concavity properties of the heat kernel. In  this respect, 
the value $\alpha = 1/2$ appears to be a critical threshold, because the heat kernel turns out to be $\alpha$-logconcave if and only if 
$\alpha \geq 1/2$. 

It is then natural to ask whether $(1/2)$-log concavity is preserved -- or even arises spontaneously --  along the Dirichlet heat flow, and whether the 
Laplacian ground state is $(1/2)$-logconcave. 
An affirmative answer to both the  the preservation and the spontaneous appearance of $(1/2)$-log concavity along the Dirichlet heat flow was given in \cite[Corollary 3.1 and Theorem 3.3]{IST20}. The authors also showed that $(1/2)$-logconcavity is, in a suitable sense,  the strongest concavity property preserved along this flow \cite[Theorem 3.2]{IST20}. 

Yet the question about the ground state of the Laplacian remained  open, as it was recently pointed  out also  in \cite{HN25}.
Aim of this paper is to answer affirmatively such question.

We emphasize that the limiting argument used by Brascamp-Lieb to deduce the logconcavity of the ground state from the preservation of  logconcavity  breaks down when one attempts to reproduce it for  $(1/2)$-logconcavity.  Indeed,   the largest parameter $\k _t$  for which
\eqref{f:alfaconc}  holds, with $\alpha = 1/2$, for the solution to the heat equation at time $t$ might, in principle, degenerate to $0$ as $t \to + \infty$, 
(see \cite[Remark 4.2]{IST20} and \cite[Section 5]{IST22}). 
Our main result establishes that such degeneracy does not occur, since we obtain that 
the ground state is $(1/2)$-logconcave:

\begin{theorem}\label{t:12log}  Let $\Om$ be an open bounded convex subset of $\R ^n$, and let $u$ be its first Dirichlet Laplacian eigenfunction, normalized so that 
$\max _{\overline \Om } u = 1$. 
Then $u$ is $(1/2)$-logconcave. More precisely,  if $\lambda _ 1 (\Om)$ and $D _\Om$ denote respectively the first Dirichlet Laplacian eigenvalue of $\Om$ and its 
diameter, 
letting 
\begin{equation}\label{f:kappa} 
 \overline \k (\Om):= \exp \Big [- \frac 3 2 \Big ( \frac{ \lambda _ 1 (\Om) D _\Om^2  }{\pi ^ 2}  -1 \Big )  \Big ]  \,,
 \end{equation} 
for every $\k \in (0,  \overline \k (\Om)]$ it holds that
$$ L _{1/2}  (\k u ( ( 1- t )  x + t  y) \geq ( 1- t) L _{1/2} (\k u (x)) + t L _{1/2} (\k u  (y)) \qquad \forall x, y \in \Om, \ \forall t  \in [0, 1]\,.$$  

\end{theorem} 

\smallskip 
 \begin{remark}\label{r:kappino} In order to estimate the size of  the constant $\overline \k (\Om)$ in \eqref{f:kappa},   observe that it is a monotone decreasing function of   the product $ \lambda _ 1 (\Om) D _\Om^2 $. 
 By the Faber-Krahn inequality and the isodiametric inequality,    this product is minimized by balls, 
while it diverges to $+ \infty$ along sequences of thinning convex domains.  Consequently, the value of $\overline k (\Om)$ is  maximized by balls, whereas it degenerates to $0$ along sequences of thinning convex domains. Thus, denoting by $B _n$ a ball of radius $1$  in  $\R^n$, and by  $j _{\nu, 1}$ the first zero of the Bessel function  $J _\nu$,  with $\nu = \frac{n}{2}- 1$,  for any 
open bounded convex subset   $\Om \subset \R ^n$   
 we have: 
\begin{equation}\label{f:kappamax}  \overline \k (\Om) \leq \overline \k ( B_n) =  \exp \Big [- \frac 3 2 \Big ( \frac{  4 \lambda _ 1 (B_n)   }{\pi ^ 2}  -1 \Big )  \Big ] =  \exp \Big [- \frac 3 2 \Big ( \frac{ 4j _{\nu, 1} ^ 2  }{\pi ^ 2}  -1 \Big )  \Big ]\,.
\end{equation}

Notice that the map  $n \mapsto \overline \k ( B_n)$ is monotone decreasing. For $n = 2$, 
we have $j _{0, 1}\simeq 2.405$, so that 
$\overline \k (B_2) \simeq 0.133$.  

Clearly,  for certain convex domains such as balls or parallelepipeds,  the threshold  $\overline \k (\Om)$  given by \eqref{f:kappa} is 
not optimal. Indeed, 
for these domains a direct  computation of the eigenfunction shows that the threshold can be increased up to $1$; namely, 
$( - \log u ) ^ {1/2}$ is convex. 
At present, we are unable to determine whether such property holds  true for every convex domain.  
 \end{remark} 

\smallskip 
The proof of Theorem \ref{t:12log} is obtained via a nonstandard application of  the
classical convex envelope method by Alvarez--Lasry--Lions \cite{ALL}.
Actually,   the equation 
satisfied by $w_\k:= (- \log (\k u ) ) ^ {1/2}$  is of the kind
$ \Delta w _\k = f (w _\k, \nabla w _\k)$  (see \eqref{f:bvw}), for a function $f$ which  does not guarantee the validity of 
the comparison principle, and does not fit the structure conditions required to apply convex envelope method (nor other classical methods such as the concavity maximum principle by Korevaar \cite{K83, CaSp}, 
or the continuity method by Caffarelli-Friedman \cite{CafFri}).

The novelty of our approach consists  in invoking
the  improved log-concavity inequality for the ground state established by Andrews-Clutterbuck in \cite[Theorem 1.5]{AC11}, 
in order to restore the applicability of the convex envelope method. 
Such inequality asserts  that the tangent function serves as a modulus of concavity for $v = -\log u$, namely
\begin{equation}\label{f:AC}
\langle \nabla v(z) - \nabla v(y), \tfrac{z-y}{|z-y|} \rangle
\geq \frac{2\pi}{D_\Omega} 
\tan\!\Big( \frac{\pi |z-y|}{2D_\Omega} \Big)
\qquad \forall z,y \in \Omega \, , \ z \neq y\,.
\end{equation}
The key idea in our proof is that  the above inequality can be used to derive
a gradient estimate for the convex envelope $w _ \k ^ {**}$ of 
$w _\k$  (see Proposition \ref{p:key}).  This estimate leads us to show that $w _ \k ^ {**}$ is a supersolution to the equation satisfied by $w _\k$, 
 for $\kappa$ below the threshold $\overline{\kappa}(\Omega)$ introduced in \eqref{f:kappa}. 
Once proved this fact, the conclusion is reached by  circumventing the lack of comparison principle, thanks to a variational argument and to the simplicity of the first eigenfunction.

\smallskip

As pointed out in Remark \ref{r:kappino}, the threshold $\overline \k (\Om)$ in Theorem \ref{t:12log} is suboptimal, at least  for some domains.
On the other hand we  are able to establish that,  even beyond this threshold, the ground state enjoys 
the following  {\it local} $(1/2)$-logconcavity property 
in a convex neighbourhood of its maximum point:

\begin{theorem} \label{t:level} Let $\Om$ and $u $ be as in Theorem \ref{t:12log}. 
For every $\k \in (0, 1)$,  
let $w_\k := ( - \log (\k u  )  ) ^ {1/2}$
and let $(w_\k)^{**}$ denote its convex envelope in $\Omega$.
Then there exists $\overline u _\k \in (0, 1)$, with
$\lim \overline u _ \k  = 1$ as $\k   \to 1 ^ -$, such that
$w_\k = (w_\k)^{**}$
on the nonempty convex subset of $\Om$ given by 
\[
\Om_ \k :=  \big \{x \in \Om \ :\  u(x) > \overline u _\k  \big \}\,.
\]
In particular, it holds that
$$ L _{1/2}  (\k u ( ( 1- t )  x + t   y) \geq ( 1-t) L _{1/2} (\k u (x)) + t L _{1/2} (\k u  (y)) \qquad \forall x, y \in \Om _\k\,, \ \forall t \in [0, 1]\,.$$  
\end{theorem} 

The proof of Theorem \ref{t:level} again relies on  the improved concavity inequality \eqref{f:AC}, this time combined with a classical gradient estimate for the ground state. 

\smallskip 
At present, it is unclear to us whether the local result given in Theorem \ref{t:level} 
might be helpful as a starting point for improving the global result given in Theorem \ref{t:12log}.

\smallskip
The paper is organized as follows: 
after providing some preliminaries in Section \ref{sec:prel}, the proofs of Theorems \ref{t:12log} and \ref{t:level} are given respectively in Sections
\ref{sec:proof1} and \ref{sec:proof2}.

\section{Preliminaries}\label{sec:prel} 

In this section we provide  some lemmas needed in the proof of Theorem \ref{t:12log}. 

We begin by quoting, without proof, a classical Li-Yau-type gradient estimate for the first Dirichlet eigenfunction, 
which can be found, for instance,
in \cite[Corollary 5.1]{sperb} (see also  \cite{LiYau, yang99}).

\begin{lemma} \label{l:yau}
 Let $\Om$ be an open bounded subset of $\R ^n$, let $u$ be its first Dirichlet Laplacian eigenfunction, normalized so that $\max _{\overline \Om } u =  1$,  and let 
$\lambda _ 1 (\Om)$ denote  the first Dirichlet Laplacian eigenvalue of $\Om$.   Then it holds that
 $$|\nabla u (x)| ^ 2 + \lambda _ 1 (\Om)  | u (x)| ^ 2 \leq \lambda _ 1  (\Om) \qquad \forall x \in \Om \,.$$ 
\end{lemma} 

Next, we provide 
two basic statements concerning
 the convex envelope. 
Recall that, if $\Om$ is an open bounded convex subset  of $\R^n$,
 for any  $w : \Om \to \R$,  its convex envelope   in $\Om$ is the function 
$w ^ { **}:\Om \to \R $ defined by 
 \begin{equation}\label{f:w**} 
w^{**} ( x):= 
\inf \Big \{ \sum_{ i = 1}^{m} t _i w(x_i) \ :\ 
m \in \N,\
t _ i\geq 0,\ x_i \in \Omega\,, \
\sum _{i = 1} ^ m t _ i = 1,\ 
\sum _{i = 1} ^ m t _ i x_i = x  \Big \}.
\end{equation}

The first statement  (Lemma \ref{l:bdr}) concerns some features of  $w^{**}$,  holding when $w$ is  lower semicontinuous and diverges to $+ \infty$ at the boundary of $\Om$.  
Although the proof relies on standard arguments, we include it for the sake of self-containedness.

\smallskip
The second statement (Lemma \ref{l:ALLp1}) is essentially a less technical version of Proposition  1 in \cite{ALL}, 
holding when $w$ is sufficiently smooth.   It
gathers some information about the second order subjet of $w ^ {**} $  at $x$, 
that we  denote by the usual notation $\subjet w^{**}(x)$. (Recall that, by definition,  it consists of all pairs 
$( p, A) \in \R ^n \times \R ^ { n\times n} _{{\rm sym}} $ for which there exists a smooth function $\varphi$ on $\R ^n$ 
such that $u - \varphi$ has a local minimum at $x$ and $(p, A) = (\nabla \varphi (x), \nabla ^ 2 \varphi (x))$.)

\begin{lemma}
\label{l:bdr}
Let $\Omega$ be an open bounded convex subset of $\R^n$,  and 
let $w\colon\Omega\to\R$ be a lower semicontinuous function
satisfying
$w(x) \to +\infty$ as $\dist(x, \partial\Omega) \to 0$.

Then $w^{**}(x) \to +\infty$ 
as $\dist(x, \partial\Omega) \to 0$.

Moreover, for every $x\in\Omega$, there exist
$t_1, \ldots, t_m > 0$, $m\leq n+1$, with $\sum_{i=1}^m t_i = 1$,
and points $x_1, \ldots, x_m \in \Omega$, such that 
\begin{equation}
\label{f:min}
w^{**}(x) = \sum_{i=1}^m t_i w(x_i),
\end{equation}
i.e., the infimum which defines $w ^ {**}$ according to \eqref{f:w**} is attained.
\end{lemma}

\begin{proof}
Let us prove first that $w^{**}(x) \to +\infty$ 
as $\dist(x, \partial\Omega) \to 0$.

Since, under our assumptions, $w$ is bounded from below,
by possibly adding a constant to $w$ it is not restrictive to 
assume that $w\geq 1$ in $\Omega$.
Assume by contradiction that there exists a sequence $\{x_h\}\subset \Omega$ and a constant $C>1$ such that
\[
\lim_{h\to +\infty} \dist (x_h, \partial\Omega) = 0,
\qquad
w^{**}(x_h) \leq C\quad \forall h\in\N.
\]
For every $h\in\N$, there exist points
$x_h^1,\ldots, x_h^{m_h}\in\Omega$
 and positive numbers $t_h^1, \ldots , t_h^{m_h}> 0$ such that
\[
\sum_{i=1}^{m_h} t_h^i = 1,
\qquad
x_h = \sum_{i=1}^{m_h} t_h^i x_h^i,
\qquad
\sum_{i=1}^{m_h} t_h^i w(x_h^i) \leq w^{**}(x_h) +1\,.
\]
Let $K := \{x\in\Omega\colon w(x) \leq C+2\}$.
By assumption, $K$ is a compact subset of $\Omega$, 
so that $\dist(K, \partial\Omega) > 0$.
For every $h\in\N$, let
\[
I_h :=\big  \{i\in\{1,\ldots,m_h\}\colon x_h^i\in K \big \},
\qquad
J_h := \big  \{i\in\{1,\ldots,m_h\}\colon x_h^i\not\in K \big \},
\qquad
\alpha_h := \sum_{i\in I_h} t_h^i\,.
\]
We have that
\[
\begin{split}
C & \geq w^{**}(x_h) \geq   \sum_{i=1}^{m_h} t_h^i w(x_h^i) -1
=
\sum_{i\in I_h} t_h^i w(x_h^i)
+ \sum_{i\in J_h} t_h^i w(x_h^i) -1
\\ & \geq
\sum_{i\in I_h} t_h^i w(x_h^i) + (C+2) (1-\alpha_h) -1\,.
\end{split}
\]
Taking into account that $w\geq 1$, from this estimate
we obtain that
\[
\alpha_h \leq \sum_{i\in I_h} t_h^i w(x_h^i)
\leq (C+2)\alpha_h - 1,
\]
i.e.,
$\alpha_h \geq 1/(C+1)$.
Finally, by the concavity of the distance function,
\[
\dist(x_h, \partial\Omega) \geq
\sum_{i=1}^{m_h} t_h^i \dist(x_h^i, \partial\Omega)
\geq
\sum_{i\in I_h} t_h^i \dist(x_h^i, \partial\Omega)
\geq \frac{1}{C+1}\, \dist(K, \partial\Omega), 
\]
in contradiction with the assumption
$\dist(x_h, \partial\Omega)\to 0$.

\smallskip
Let us prove now the second part of the lemma.
Let $x\in\Omega$.
By Carathéodory's theorem, for every $h\in\N$ there exist points
$x^h_1, \ldots, x^h_{n+1} \in \Omega$
and non-negative numbers $t^ h_1, \ldots, t^ h_{n+1}$, such that
\begin{equation}
\label{f:approx}
\sum_{i=1}^{n+1} t^ h_i = 1, 
\qquad
w^{**}(x) \geq \sum_{i=1}^{n+1} t^ h_i w(x^ h_i) - \frac{1}{h}\,.
\end{equation}
Since $\Omega$ is bounded, we can extract a subsequence (not relabeled) such that
\[
\lim_{h \to +\infty} x^ h_i= x_ i \in \overline{\Omega},
\qquad
\lim_{h \to +\infty}  t^ h_i= t_ i \in [0,1]\,.
\]
Clearly, $t_i\geq 0$ and $\sum_{i=1}^{n+1} t_i = 1$.
Without loss of generality, we can assume that 
there exists $m\in \{1, \ldots, n+1\}$ such that 
$t_1, \ldots, t_m > 0$ and $t_i = 0$ for $i> m$, $i\leq n+1$.

Since our assumptions  imply in particular that $w$ is nonnegative,  
by \eqref{f:approx}
we have that, for every $i = 1, \ldots, n+1$,
\[
w^{**}(x)
\geq t^ h_i w(x^h_i) - \frac{1}{h}\,.
\]
Now, if $x_i\in\partial\Omega$ for some $i\in \{1,\ldots, n+1\}$, 
by passing to the limit as $h\to +\infty$ in the above inequality
and taking into account that  $w(x^ h_i) \to +\infty$,
we deduce that $t_i = 0$, so that $i\geq m+1$.
As a consequence, the points $x_1, \ldots, x_m$ belong to $\Omega$,
and $x = \sum_{i=1}^m t_i x_i$.
Finally, by  using the definition of $w^{**}$ and the inequality in \eqref{f:approx}, we obtain
\[
\sum_{i=1}^m t_i w(x_i)
\geq w^{**}(x) \geq
\liminf_{h\to +\infty} \sum_{i=1}^{n+1} t^ h_i w(x^ h_i)
\geq 
\liminf_{h\to +\infty} \sum_{i=1}^{m} t^ h_i w(x^ h_i)
\geq
\sum_{i=1}^{m} t_i w(x_i),
\]
so that \eqref{f:min} follows.
\end{proof}

\bigskip
\begin{lemma}
\label{l:ALLp1}
Let $\Omega$ be an open, bounded, convex subset of $\R^n$,  and 
let $w\in C^2(\Omega)$ be such that $w(x) \to +\infty$ as $\dist(x, \partial\Omega) \to 0$.

Given $x\in\Omega$, let $t_1, \ldots, t_m > 0$, $m\leq n+1$, with $\sum_{i=1}^m t_i = 1$,
and let $x_1, \ldots, x_m \in \Omega$ be optimal for the infimum defining $w ^ {**} (x)$ according to \eqref{f:w**}.

If $(p,A) \in \subjet   w^{**}(x)$,  then 
the following properties hold:
\begin{itemize}
\item[(i)] $\nabla w(x_i) = p$
for every $i = 1,\ldots,m$;
\smallskip

\item[(ii)] the matrices $A_i := \nabla ^2 w(x_i)$ are positive semidefinite for every $i = 1,\ldots,m$;
\smallskip

\item[(iii)] if all the matrices $A_i$, $i=1,\ldots,m$, are positive definite,
then
\[
A \leq \left(\sum_{i=1}^m t_i A_i^{-1}\right)^{-1},
\qquad
\Tr(A) \leq 
\left(\sum_{i=1}^m t_i \Tr (A_i^{-1})\right)^{-1};
\]

\item[(iv)] if at least one of the matrices $A_i$ is degenerate, then
$\Tr(A) \leq 0$.
\end{itemize}
\end{lemma}

\begin{proof}
The properties (i)--(iv) were already discussed  in \cite[p.~272]{ALL}. 
However,  for the sake of clarity, we provide some additional details concerning items (iii) and (iv).
In \cite{ALL} the authors, after establishing (i) and (ii), went on to show that
\begin{equation}
\label{f:minA}
\pscal{Ah}{h} \leq
\sum_{i=1}^m t_i \pscal{A_i h_i}{h_i}
\end{equation}
for every $h, h_1, \ldots, h_m\in \R^n$ such that
$h = \sum_{i=1}^m t_i h_i$.
If all the  matrices $A_i$ are positive definite,  then for every fixed $h$
we can minimize the right-hand side of \eqref{f:minA}
under the constraint $h = \sum_{i=1}^m t_i h_i$,
obtaining  \[
h_i = A_i^{-1} \left(\sum_{i=1}^m t_i A_i^{-1}\right)^{-1}\!\!\! \!\! h ,
\qquad i = 1, \ldots m\,.
\]
Replacing these expressions for  $h _i$ into \eqref{f:minA} yields
the first inequality in (iii).

The second inequality then follows from the concavity
of the map $Q \mapsto [\Tr(Q^{-1})]^{-1}$ on the class of
positive definite symmetric matrices, proved in the Appendix of \cite{ALL}. 

To prove (iv), let $A_{i,N} := A_i +\frac { I }{  N}$, where $I$ denotes the $n\times n$ identity matrix. 
From \eqref{f:minA} we have 
\[
\pscal{Ah}{h} \leq
\sum_{i=1}^m t_i \pscal{A_{i,N} h_i}{h_i}\,.
\]
By performing the same minimization of the right-hand side as above,
we deduce that (iii) holds with the matrices $A_{i,N}$ in place of $A_i$.
If $A_k$ is degenerate for some $k\in \{1,\ldots,m\}$,
then $\Tr(A_{k,N}^{-1}) \to +\infty$ as $N\to +\infty$. Therefore, 
from the second inequality in (iii), we conclude that
$\Tr(A) \leq 0$. 
\end{proof}

\bigskip\bigskip
\section{Proof of Theorem \ref{t:12log}} \label{sec:proof1}

Throughout this section, we assume without any further mention that 
$u$ is the first Dirichlet Laplacian eigenfunction of  an open bounded convex subset $\Om$ of $\R^n$, 
normalized so that $\max_{\overline \Om } u  = 1$. For any $\k \in (0, 1)$, 
we set 
\begin{equation}\label{f:norm} 
w  _\k:= (- \log (\k u) ) ^ { 1/2} \qquad \text{ in } \Om\,. 
\end{equation}   

Observe that $\min_\Omega w_\k = (- \log \k ) ^ { 1/2} > 0$. 
Moreover, some straightforward  computations show that $w _\k $ is a (classical) solution to 
\begin{equation}\label{f:bvw}
 - \Delta w + \frac{1}{w} \big [ ( 2 w ^ 2 -1 ) |\nabla w| ^ 2 +\frac{\lambda _ 1 (\Om) }{2} \big ] = 0 \qquad \text{ in } \Om\,. 
 \end{equation}

We shall denote by $w _\k ^ { **}$ the convex envelope of $w _\k$ defined according to \eqref{f:w**}, 
and by $u _k$ the function defined in $\Om$ by 
\begin{equation}\label{f:normu} 
u_\k := \exp (  - \log \k -  ( w _\k ^ {**} ) ^ 2  ) 
\,, \qquad \text{i.e.},\  w_\k ^ {**} = (- \log (\k u_\k) ) ^ { 1/2}\, .
\end{equation}

We now outline of the main steps of our proof: 
\smallskip

\begin{itemize} 
\item[--]  In Proposition \ref{p:key} we establish a gradient estimate for $w _\k ^ { **}$, valid for every $\k \in (0, 1)$,  as a consequence of the improved logconcavity inequality \eqref{f:AC}.
\smallskip

\item[--] In Proposition  \ref{p:viscosuper}    we prove that, provided  $\k \in (0, \overline \k (\Om)]$, $w _\k ^ { **}$ is a viscosity supersolution of equation \eqref{f:bvw} in $\Om$. 
\smallskip

\item[--]  In Proposition \ref{p:salani} we show that, for every $\k \in (0, \overline \k (\Om)]$, $u _\k$   belongs to $H ^ 1 _0 (\Om)$,
attains the same maximum value $1$ as $u$, and satisfies $-\Delta  u _\k   \leq  \lambda _ 1 (\Om )   u  _\k  $ a.e. in  $ \Om$.
 \end{itemize} 

\smallskip 

We then deduce Theorem  \ref{t:12log} by exploiting the simplicity of the first eigenfunction. 
\medskip 

\begin{proposition}\label{p:key}  Let $\k \in (0, 1)$ and assume that, at some point $x \in \Om$, one has $w _\k   ^ { **}  (x) < w _\k  (x) $. 
Let $x _i \in \Om$, $i = 1, \dots, m$  (with $m \geq 2$) be points where the infimum which defines   $w _\k   ^ { **}  (x)$ according to
\eqref{f:w**} is attained, and let $p_\k$ 
denote  the common value of  $\nabla w _\k$ at any of the points   $x_i$, as given by Lemma  \ref{l:ALLp1}  (i). 
Then 
 $$|p_\k | ^ 2 \geq  \frac{\pi ^ 2}{2 D _\Om ^ 2} \,.$$   
\end{proposition}

\begin{proof} 
Since the function $w _\k ^ 2$ satisfies the improved log-concavity estimate \eqref{f:AC},  for every pair of distinct points $z, y$ in $\Om$, we have
  $$
 \langle  w _\k ( z) \nabla w _\k ( z) - w _\k (y) \nabla w _\k ( y), \frac{ z-y }{|z-y |}  \rangle 
\geq \frac{ \pi }{D_\Om} \tan \Big (   \frac{\pi  |z-y| }{2 D_\Om}  \Big ) \,.
$$
Pick  two distinct points $x_i$ and $x _j$ among those at which the infimum which defines   $w _\k   ^ { **}  (x)$ according to
\eqref{f:w**} is attained, and set 
 $p _\k = \nabla w _\k ( x_i) = \nabla w _k ( x_j)$.
  Applying the above estimate,  we obtain
 $$
 \langle   (w _\k  ( x_i )- w_\k   ( x_j) ) p  _\k , \frac{ x_i-x_j }{|x_i-x_j|}  \rangle 
\geq \frac{ \pi }{D_\Om} \tan \Big (   \frac{\pi  |x_i-x_j| }{2 D_\Om}  \Big )  \,.
$$
Since $w_\k ^{**}   $ is affine on the segment $[x_i, x_j]$, we have
  $$w_\k   ( x_i )- w_\k   ( x_j)  = \langle p_\k, x_i - x_j \rangle  \,.$$ 
By substituting this identity into the previous inequality, and then by applying Cauchy-Schwarz inequality 
 and the elementary inequality $\tan s \geq s$ for $s\in [0, \frac{\pi}{2}) $, 
we obtain
 \[
|p_\k| ^ 2 \geq  \frac{ \pi }{D_\Om} \frac{1}{ |x_i-x_j| }   \tan \Big (   \frac{\pi  |x_i-x_j| }{2 D_\Om}  \Big )   
\geq   \frac{\pi ^ 2}{2 D _\Om ^ 2} \,.
\qedhere 
\]
\end{proof}

\bigskip
 
\begin{proposition}\label{p:viscosuper}   
For any $\k \in (0, \overline k (\Om)]$, the convex envelope $w  _\k ^{**}$ is a viscosity  supersolution  to   equation \eqref{f:bvw}. 
  \end{proposition} 

\begin{proof}
Let $\k \in (0, \overline k (\Om)]$ be fixed. We have to show that, setting  
$$F (s , \xi, X):= - {\rm {\rm Tr} } ( X)  + \frac{1}{s} \big [ ( 2 s ^ 2 -1 ) |\xi| ^ 2 +\frac{\lambda _ 1 (\Om) }{2} \big ] \,,$$ 
it holds that
\begin{equation}\label{f:goal1}
F ( w _\k ^{**}(x), p, A ) \geq 0 \qquad \forall x \in \Om \,, \quad \forall (p, A) \in \subjet w _\k  ^ { ** } ( x) \,.
\end{equation}

Let $x \in \Om$ be fixed. 
We first observe that \eqref{f:goal1} holds trivially  if $w _\k ^{**}(x) = w _\k ( x)$. Indeed in this case we have  
$ \subjet w _\k  ^ { ** } ( x) \subseteq \subjet w _\k   ( x)$, so that 
$(p, A) \in \subjet w _\k   ( x)$, and thus \eqref{f:goal1} holds because $w _\k$ is a classical solution, hence a viscosity solution,  to equation \eqref{f:bvw}. 
Observe that this case happens, in particular, when $p=0$, because in this case $x$ must coincide with the unique minimum point of $w_\k$,
i.e., the unique maximum point of $u$.

Thus, in the remaining of the proof we are going to assume that $w _\k ^{**}(x) < w _\k ( x)$ and $p\neq 0$.

By Lemma \ref{l:bdr}, there exist $t ^\k  _i \in (0, 1)$ and  $x ^\k _i \in \Om$, $i = 1, \dots, m_\k$ (with $m _\k \geq 2$), 
such that
\begin{equation}\label{f:env}
\sum _{i=1}^{m _\k} t ^\k  _i =1\,, \qquad
x = \sum _{i=1}
 ^{m _\k} t ^\k  _i x ^\k  _i\,, \qquad w^{**} _\k  (x)  = \sum _ { i = 1} ^{ m_\k} t ^\k   _i w _\k  ( x ^\k _i) \,.
\end{equation} 
 
Throughout the proof, the parameter
$\k \in (0, \overline \k (\Om)]$, the point $x \in \Om$, and accordingly the
numbers $t^\k _i $ and  the points $x^\k _i $ in \eqref{f:env}  will remain fixed. 
Taking also into account that,  by Lemma \ref{l:ALLp1} (ii), the vector  $\nabla  w  _\k ( x_i)$  is  independent of  $i \in \{ 1, \dots, m _\k\}$,
we are going to write for brevity  
$$w ^{**} := w^{**} _\k  (x)$$
and, for every $i = 1, \dots, m: = m _\k$ (with $m \geq 2$), 
$$t_i:= t ^ \k _ i \, , \qquad w^i:= w  _\k ( x_i),    \qquad p:= \nabla  w  _\k ( x_i)\ , \qquad A _{ i}: = \nabla ^ 2 w _\k  ( x_i) \,.$$ 
Moreover, we set 
 \begin{equation}\label{f:defa} 
 a:=  \frac{\lambda _1 (\Om)}{4| p | ^ 2 } - \frac 1 2 \,,
 \end{equation}
so that
$$\begin{aligned}
& F (w^{**}, p, A )  = - \Tr (A)+  2 |p| ^ 2 \frac{(w^{**} ) ^ 2 + a}{w ^ {**}}\,,
\\ 
& F (w^i, p, A_i )  =- \Tr (A_i) + 2 |p| ^ 2 \frac{(w^{i} ) ^ 2 + a}{w ^ {i}} \,.  
 \end{aligned}$$

By Lemma \ref{l:ALLp1} (ii), we know that all the matrices $A_i$ are positive semi-definite.  
  Let us first prove the inequality \eqref{f:goal1} when  all of them are 
positive definite. 
In this case,  by Lemma \ref{l:ALLp1}  (iii),  it holds that
 $$\Tr(A) \leq 
\left(\sum_{i=1}^m t_i \Tr (A_i^{-1})\right)^{-1}\,. $$ 
 Hence,  $$\begin{aligned} 
& F (w^{**}, p, A )  
 =  - {\rm Tr} ( A)  +  2 |p| ^ 2 \frac{(w^{**} ) ^ 2 + a}{w ^ {**}}   \\ 
 & \geq   - \Big \{  \sum _ { i = 1} ^ m  t _i  \big [  {\rm Tr} (A _i )   \big ]   ^ { -1}   \Big \} ^ { -1}  
 + 2 |p| ^ 2 \frac{(w^{**} ) ^ 2 + a}{w ^ {**}}    \\ 
 & =   - \Big \{  \sum _ { i = 1} ^ m  t _i  
 \Big [  2 |p| ^ 2 \frac{(w^{i} ) ^ 2 + a}{w ^ {i}}   \Big ]      ^ { -1}   \Big \} ^ { -1}   + 2 |p| ^ 2 \frac{(w^{**} ) ^ 2 + a}{w ^ {**}}     \,,
   \end{aligned} 
 $$ 
 where in the last equality we have used the identity
  \begin{equation}\label{f:equazz} 
   F (w^i, p, A_i ) = 0 \qquad \forall i = 1, \dots, m, 
 \end{equation}
  holding since $w _\k $ satisfies the pde  \eqref{f:bvw}
 at the points $x _i$'s.

 We infer that the target inequality \eqref{f:goal1} is fulfilled provided
 \begin{equation}\label{f:goal2} \frac{w^{**}}{(w^{**}) ^ 2  + a } 
 \leq \sum _ { i = 1} ^ m  t _i  
 \frac{w^i} { (w^{i} ) ^ 2  + a   }  \,.
 \end{equation} 
We now observe that, since we are assuming that the matrices $A_i$ are positive definite, and since
the equalities \eqref{f:equazz}  hold, we have 

 $$0 < {\rm Tr} ( A_i) = 2 |p | ^ 2 \, \frac {(w^ {i} ) ^ 2  + a  }{w^ i}   \qquad \forall i = 1, \dots, m.$$ 
 We deduce that 
 \begin{equation}\label{f:app} 
 w ^ i \in \mathcal R:=  \big \{  s \in (0,   + \infty)  \text{ such that } s ^ 2 > - a  \big \} \qquad \forall i = 1, \dots, m \,.
 \end{equation} Hence, we see that
the inequality \eqref{f:goal2} is satisfied provided  the values $\{w^ 1, \dots, w ^m, w ^{**}\}$ fall in the subregion $\mathcal R_0$ of $\mathcal R$ where the function 
 $$\varphi _ a ( s):= \frac{s}{ s ^ 2 + a  }\,,$$ 
 
is convex. 
 Since 
 $$\varphi _ a''  ( s):= \frac{2s (s^2- 3a)}{ (s ^ 2 + a ) ^ 3  } \,,$$ 
 we have
 $${\mathcal R}_0=  \mathcal R \cap \{ s^2 \geq 3 a \}  \,.$$ 
 So we distinguish two cases:
 
 \smallskip
 {\it Case 1}:  $a \leq 0$. In this case we have 
 $${\mathcal R}_0=  \mathcal R = \{ s > \sqrt { - a \}}\,. $$ 
 By \eqref{f:app}, we have that, for every $i = 1, \dots, m$,  the values $w ^ i$ lie in $\mathcal R _0$; by the convexity of $\mathcal R _0$, also  $w ^{**}$ lie in $\mathcal R _0$. Thus
 the required inequality \eqref{f:goal2} is satisfied, and our proof is achieved.
 
 \smallskip 
  {\it Case 2}:  $a >0$. In this case,  imposing that $w ^ i \in \mathcal R _0$ amounts to ask that 
\begin{equation}\label{f:stimap}
|p| ^2\geq \frac{\lambda _ 1 (\Om) }{4} \frac{1}{\frac{ (w ^ i) ^ 2}{3} + \frac {1}{2}  } \qquad \forall i = 1, \dots, m \,.
\end{equation}
We claim that these inequalities are fulfilled thanks to the assumption that the constant $\kappa$ in \eqref{f:norm} does not exceed $\overline \kappa (\Om)$. 
Indeed, since $m \geq 2$, we are in a position to apply Proposition \ref{p:key}: it implies that
a sufficient condition for the validity of the inequalities in  \eqref{f:stimap} is 
 $$ \frac{\pi ^ 2}{2 D _\Om ^ 2} \geq   \frac{\lambda _ 1 (\Om) }{4} \frac{1}{\frac{ (w ^ i ) ^ 2}{3} + \frac {1}{2}  } \qquad \forall i = 1, \dots, m \,, $$ 
 or equivalently 
 $$( w^i ) ^ 2 =  - \log (\k u(x_i)) )   \geq  \frac 3 2 \Big ( \frac{ \lambda _ 1 (\Om) D _\Om^2  }{\pi ^ 2}  -1 \Big )  \qquad \forall i = 1, \dots, m   \,.$$ 
Such inequalities are satisfied  because  $0 < u (x_i) \leq 1$ and  $\kappa  \leq \overline \kappa (\Omega)$.  

 \smallskip
It remains to prove that  the inequality 
\eqref{f:goal1}  is satisfied also when some of the $A_i$ is degenerate. In this case, by  Lemma \ref{l:ALLp1} (iv), we know that ${\rm Tr} (A) \leq 0$. Hence,  
 $$
F (w^{**}, p, A ) 
 =  - {\rm Tr} ( A)  +  2|p|^ 2 \frac{( w^ {**} ) ^2 +a   }{w^{**}  }  \geq  2|p|^ 2 \frac{( w^ {**} ) ^2 +a   }{w^{**}  } \,.
$$
To conclude we observe that: if $a \geq 0$, we have immediately   $
 F (w^{**}, p, A ) \geq 0$; on the other hand, if $a <0$, we have that 
 the map $s \mapsto \frac{ s ^ 2 + a}{s} = s - \frac{|a|}{s}$ is concave in $(0, + \infty)$, and hence
 \[
 F (w^{**}, p, A )\geq  2|p|^ 2 \frac{( w^ {**} ) ^2 +a   }{w^{**}  }  \geq 2|p|^ 2  \sum _ {i= 1} ^m t_i \frac{ (w^ i )  ^ 2 + a}{w^ i}  =    \sum _ {i= 1} ^ m  t_i \Tr (A_i)  \geq 0 \,.
 \qedhere
\]
\end{proof}

\bigskip

\begin{proposition}\label{p:salani}   Let $\k \in (0, \overline \k (\Om)]$, and let $ u_\k$ be defined by  \eqref{f:normu}. 
Then: 
\begin{itemize}
\item[(i)]  $u_\k$ is a viscosity subsolution to $-\Delta u = \lambda _ 1 (\Om ) u$ in $\Om$; 
\smallskip 

\item[(ii)] it holds that $-\Delta   u_\k  (x) \leq  \lambda _ 1 (\Om )   u _\k (x) $ for a.e. $x \in \Om$;
\smallskip 

\item[(iii)]    we have that $\max _{\overline \Om } u _ \k   = 1$,  and $ u_\k  \in H ^1 _0 (\Om)$.
\end{itemize} 
\end{proposition}

\proof   (i) By definition of viscosity subsolution, we have to show that, for any $x \in \Om$, if $\varphi$ is a smooth function  
which touches $ u_\k$ from from above  at $x$ (i.e.,  such that $ u_\k(x) = \varphi(x)$ and  $ u_\k (y) \leq \varphi (y)$ for every $y\in\Omega$),
it holds that
\begin{equation}\label{f:conditionsub} - \Delta \varphi (x)  \leq \lambda _ 1 (\Om) \varphi  (x) \,.
\end{equation}
From the second equality in \eqref{f:normu}, we see that the function $\psi :=(- \log (\k  \varphi) ) ^ {1/2}  $ touches $w _\k   ^{**} $ from below at  $x$ 
  (i.e.,   $w_\k  ^ {**} (x) = \psi(x)$ and  $ w _\k  ^ {**}  (y)  \geq   \psi (y)$ for every $y\in\Omega$). On the other hand, 
  it is straightforward to check that that the inequality  \eqref{f:conditionsub} is equivalent to 
  \begin{equation}\label{f:conditionsuper}  
    - \Delta \psi  (x) + \frac{1}{\psi (x) } \big [ ( 2 \psi  (x) ^ 2 -1 ) |\nabla \psi (x)| ^ 2 +\frac{\lambda _ 1 (\Om) }{2} \big ] \geq 0  \,,\end{equation} 
The above inequality holds true because,  by Proposition \ref{p:viscosuper}, $w _ \k ^{**}$ is a viscosity  supersolution  to  equation  \eqref{f:bvw}.

\smallskip
(ii) By statement  (i) and the definition of viscosity subsolution, it follows that the inequality   $-\Delta   u_\k  (x) \leq  \lambda _ 1 (\Om )   u_\k (x) $  is satisfied at every point where $ u_\k$ is twice differentiable. Then statement (ii) follows taking into account that, by Alexandrov Theorem, the function $ w _ \k   ^ {**}$ is twice differentiable almost everywhere in $\Om$, and hence the same holds true for $u_\k = \exp (  - \log \k -  ( w _\k ^ {**} ) ^ 2  )$. 

\smallskip
(iii) By definition of convex envelope, we have  that  
$$w _ \k(x) \geq w  _\k ^ { **} (x) \geq \min _{\overline \Om}  w _\k   \qquad \forall x \in \Om \,,$$
which ensures in particular that 
\begin{equation}\label{f:minconv} \min _{\overline \Om}  w  _\k ^ { **}   = 
 \min _{\overline \Om}   (w _\k )\
  (=  \sqrt{  - \log  \k})\,.
  \end{equation}
 Then, from its definition, the function $u_\k$  satisfies 
\begin{equation} \label{f:b1} 
\max _{\overline \Om}  u _ \k   = 1  \,.\end{equation}
Let us show  that $ u_\k$ is Lipschitz continuous in $\Omega$, with 
\begin{equation} \label{f:b2} |\nabla  u_\k ( x)  | \leq \sqrt {\lambda _ 1 (\Om)} \qquad \forall x \in \Om\,.\end{equation}

From the definition of $u_\k$, we see that 
\begin{equation}\label{f:nabla} 
|\nabla  u _\k | = 2 w _\k ^ {**} \exp (  - \log \k -  ( w _\k ^ {**} ) ^ 2  ) |  \nabla w _\k ^ {**}  |  = 2 \Phi  _\k ( w _\k ^ {**} ) |  \nabla w _\k ^ {**}  | \,,
\end{equation}   
where we have set $\Phi _\k ( s) := s \exp ( - \log \k -  s^ 2  ) $; in a similar way, we have that
 \begin{equation}\label{f:nabla2} 
|\nabla  u | = 2 w _\k  \exp (  - \log \k -  ( w _\k  ) ^ 2  ) |  \nabla w _\k  |  = 2 \Phi  _\k ( w _\k  ) |  \nabla w _\k  | \,.
\end{equation}    

We have $\Phi_\k  '( s) = (1-2 s ^2)  \exp ( - \log \k -  s^ 2  ) $, so that $\Phi_\k$ is monotone decreasing on the half-line 
 $[\frac{ 1}{\sqrt 2} , + \infty)$.  Notice that, thanks to the choice $\k \leq \overline \k ( \Om)$, and to the estimate 
 \eqref{f:kappamax}, we have 
$\k \leq \frac {1}{\sqrt e}$,  so that the half-line 
 $[\sqrt{ - \log \k} , + \infty)$  is contained into
 $[\frac{ 1}{\sqrt 2} , + \infty)$.

 Now we observe that the gradient of the convex envelope enjoys the following property:   
 for every $x \in \Om$, there exists 
 $y _x \in \Om$ such that $w _ \k ( y_x) \leq w _\k  ^ {**} (x)$ and 
$\nabla w _\k ^  {**} (x) = \nabla w _ \k ( y _x)$.   This follows immediately  from Lemma \ref{l:ALLp1}: it is enough  observe that, in the family of points $x_i$ where the infimum defining  $w _\k ^  {**} (x)$ according to \eqref{f:w**}   is attained, 
there exist at least one of them such that  $w _ \k ( x_i) \leq w _\k  ^ {**} (x)$, and to choose $y _ x = x_i$ 
(notice that it may occur that 
$y _x =x$ itself, in case $w _\k  ^ {**} (x) = w _\k  (x)$).

Then, by comparing \eqref{f:nabla} written at $x$ and \eqref{f:nabla2} written at $y _x$,   we have 
$$    |\nabla u _\k   (x)   |        = \frac{\Phi  _\k ( w _\k ^ { **}  (x)  )  } {\Phi _\k ( w _\k (y_x)  ) }  \,  |\nabla u   ( y_ x) |  \leq |\nabla u   ( y_ x) | \leq \sqrt {\lambda _ 1 (\Om)}\,,$$
where the first inequality holds because $w _ \k ( y_x) \leq w _\k  ^ {**} (x)$ and  $\Phi _\k$ is monotone decreasing on the half-line  $[\sqrt{ - \log \k} , + \infty)$, while   the second inequality follows from Lemma \ref{l:yau}.  

We have thus proved  the Lipschitz estimate \eqref{f:b2}, which combined with \eqref{f:b1} ensures in particular that $u _\k \in H ^ 1 (\Om)$. Finally, the fact that $u_\k \in H ^ 1 _0 (\Om)$  follows by noticing that, as $x \to \partial \Om$, we have 
 $w _\k  (x) \to + \infty$, which implies by Lemma \ref{l:bdr} that 
also 
$w _\k ^ { **} (x) \to + \infty$.  
   \qed

 \bigskip
\bigskip\bigskip

{\bf Proof of Theorem \ref{t:12log}}.   Let $ u_\k$ be defined by  \eqref{f:normu}.   Multiplying by $ u_\k $ the inequality holding 
by Proposition \ref{p:salani} (ii), we obtain 
$$  u _\k ( x) \Delta   u_\k  (x) \leq  \lambda _ 1 (\Om )  u _\k  ^2 (x)  \qquad \text{ a.e. in }  \Om\,.$$  
We now integrate over $\Omega$. Taking into account that  $ u_\k$ belongs to $H ^1 _0 (\Om)$ (cf.\ Proposition \ref{p:salani} (iii)), we obtain  
$$\int _\Om |\nabla  u _\k | ^ 2 \leq \lambda _ 1 (\Om) \int _\Om | u_\k | ^ 2 \,.$$
From the variational characterization of $\lambda _ 1 (\Om)$ as the infimum of the Rayleigh quotient over $H ^ 1 _0 (\Om)$, and from the simplicity of the first eigenfunction, we deduce that $u_\k$ is a multiple of $u$, i.e., there exists some $ t \in \R$ such that 
$$ u_\k (x) = t u  (x)   \qquad \text{ in } \Om\,.$$
Finally we notice that, by Proposition \ref{p:salani} (iii), it holds  
$\max _{\overline \Om}  u _ \k   = 1$. Since the same condition holds true by assumption for $u$, 
we have  necessarily $t= 1$, which yields  $ u _\k = u$, and, in turn, $w _\k = w _\k ^ { **}$. 
\qed 

\section{Proof of Theorem \ref{t:level}} \label{sec:proof2}

We keep the notation fixed at the beginning of the previous section: for any $\k \in (0, 1)$,  we consider the function $w _\k$ given by 
\eqref{f:norm}. Then we define 
$\overline w _ \k$ as the unique number in  $( \sqrt{ - \log \k}   , + \infty)$  
such the function 
$$\Psi _\k (s ):= \frac {\k ^ 2 e ^ { 2 s ^ 2  } -1} {2 s ^ 2 } $$ 
satisfies
$$\Psi _ \k (   \overline  w _\k ) = \frac{\pi ^ 2}{\lambda _ 1 (\Om) D _\Om ^ 2 } \,.$$  
Notice that $\overline w _ \k$ is well-defined since $\Psi _\k$ is nonnegative and strictly monotone increasing
on the half-line $( \sqrt{ - \log \k  }   , + \infty)$, with 
$$\Psi _k  \big ( \sqrt{ - \log \k  } \big ) = 0 \qquad \text{ and } \qquad \lim _{s \to + \infty} \Psi _\k (s) = + \infty\,,$$ 
see Figure \ref{fig}.  
Then, since  
$$\min _{\Om}   w _\k 
  =  \sqrt{  - \log  \k} \qquad \text{ and } \qquad   w_\k (x) \to +\infty \text {  as  } \dist(x, \partial\Omega) \to 0\,,$$
the subset $\Om _\k$ of $\Om$ defined by  
$$\Om _\k := \big \{ x \in \Om \ :\ w _\k ( x) <  \overline  w _\k     \big \}  $$ 
is nonempty.  Moreover,   since
$$\Om _\k  = \Big \{ x \in \Om \ :\ u  ( x) >  \overline u _\k := \frac{1}{\k} \exp (-  \overline w _\k ^2 )  \Big \} \,,$$ 
by  the log-concavity of $u$ we have that $\Om _k$ is a convex set. 

Notice also that, since  the value $\frac{\pi ^ 2}{\lambda _ 1 (\Om) D _\Om ^ 2 } $ belongs to $(0, 1)$ (cf. Remark \ref{r:kappino}), and since
\[
\lim _{\k \to 1 ^-} \Psi _\k (s )= \Psi _1 ( s)= \frac { e ^ { 2 s ^ 2  } -1} {2 s ^ 2 } \,,
\qquad\forall s>0,
\] 
 we have that 
 $\lim _{\k \to 1 ^-} \overline w _\k = 0$. Consequently,  
 $\lim _{\k \to 1 ^-} \overline u _\k = 1$, 
 namely the convex sets $\Om _\k$ shrink to
a singleton (the maximum point of $u$)  
as $\k \nearrow 1$. 

 \begin{figure} 
 \center
 \includegraphics[width=8cm]{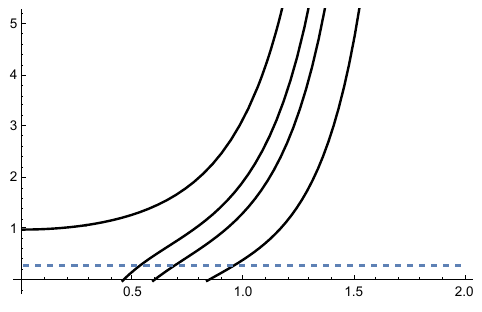}
 \caption{ Plot of the maps $s \mapsto \Psi _\k ( s)$ (from the right to the left, for $\k= \frac{1}{2},  {\frac{1}{ \sqrt  2}}, {\frac{  \sqrt 2}{  \sqrt 3}}, 1$),  and in dashed line of the constant map $s \mapsto  \frac{\pi ^ 2}{\lambda _ 1 (\Om) D _\Om ^ 2 } $.}
 \label{fig}   
 \end{figure} 

Let us prove that, denoting by $(w _\k )  ^{**}$ the convex envelope of $w _\k$ in $\Om$, it holds that
\begin{equation}\label{f:invomega}
w _\k   =  (w _\k )  ^{**}  \qquad \text{ in } \ \Om _\k \,.
\end{equation}  
Such equality ensures in particular that
the function $w _\k$  is convex on $\Om _k$.

Assume by contradiction that \eqref{f:invomega} is false at some point $x \in \Om _\k$. Then we can find numbers  $t _ i \geq 0$ and points
 $x_i \in \Om$, $i = 1, \dots, m$, such that $\sum _{i = 1} ^ m t _ i = 1$, $\sum _{i = 1} ^ m t _ i x_i = x$, and 
  $$\sum _ { i = 1} ^ m t _i w _\k ( x_i) < w _\k ( x)\,.$$ 
  We observe that, for at least one index $i _0 \in \{1, \dots, m\}$, we have  $x_{i_0}  \in \Om _\k$. 
  In fact, otherwise it would be 
  $$w _\k ( x_i ) \geq \overline w _ \k  \qquad \forall i = 1, \dots, m \,,$$ 
  which in turn would give 
  $$(w _\k   ) ^ { **} ( x) =   \sum _ { i = 1} ^ m t _i w _\k ( x_i)  \geq \overline w _ \k  \,,$$
  against the fact that $x \in \Om _ \k$. 
  
  We apply the inequality given by Lemma \ref{l:yau} at the point $x_{i_0}$.  It gives: 
  
 $$|\nabla (\k  u ) (x_{i_0} )| ^ 2 + \lambda _ 1 (\Om)  | (\k u  ) (x _{i_0} )| ^ 2 \leq \k ^ 2 \lambda _ 1  (\Om) \,,$$  
  or equivalently, in terms of the function $w _\k$: 
  $$4 w _\k ^2  (x_{i_0} )  |\nabla w _\k (x_{i_0} )| ^ 2 e ^ { - 2 w _\k  ^2 (x_{i_0} ) } + \lambda _ 1 (\Om) e ^ { - 2 w _\k  ^2 (x_{i_0} ) } \leq \k ^ 2 \lambda _ 1  (\Om) \,.$$ 
  
  Thus, letting 
  $$p _\k:=  \nabla w _\k (x_{i_0} ) \,,$$ 
  we obtain the upper bound inequality 
  $$| p_\k| ^ 2 \leq  \lambda _ 1  (\Om)  \frac{   \k ^ 2   e ^ {  2 w _\k  ^2 (x_{i_0} ) } -  1 }{  4 w _\k ^2  (x_{i_0} ) }  = \frac {\lambda _ 1 (\Om)} {2} \Psi _\k (w _\k ( x_ {i_0}) )<  \frac {\lambda _ 1 (\Om)} {2} \Psi _ \k (   \overline  w _\k ) =  \frac{\pi ^ 2}{2 D _\Om ^ 2 }  \,,$$ 
  which contradicts Proposition \ref{p:key}.  \qed

\def\cprime{$'$}

\end{document}